\documentclass[11pt]{amsart}
\usepackage{amssymb}
\usepackage[mathscr]{euscript}

\newcommand{\mbf}{\mathbf}                    
\newcommand{\msc}{\mathscr}

\usepackage{graphicx}

\newtheorem{thm}{Theorem} 
\newtheorem{cor}[thm]{Corollary}
\newtheorem{lm}[thm]{Lemma}

\newcommand{\la}{\langle}
\newcommand{\ra}{\rangle}
\newcommand{\op}{\operatorname}
\newcommand{\join}{\vee}

\renewcommand{\Join}{\bigvee}

\newcommand{\id}{\downarrow\!}

\newcommand{\vare}{\varepsilon}

\newcommand{\ifff}{\Longleftrightarrow}

\newcommand{\RA}{\Rightarrow}
\newcommand{\imp}{\!\implies\!}

\newcommand{\A}{\mathbf A}

\newcommand{\E}{\mathscr E}
\newcommand{\F}{\mathbf F}
\newcommand{\G}{\mathbf G}

\newcommand{\K}{\mathbf K}
\renewcommand{\L}{\mathbf L}
\newcommand{\M}{\mathbf M}
\renewcommand{\S}{\mathbf S}
\newcommand{\T}{\mathbf T}

\date{\today}

\begin{document}

\title[Q-lattices as congruence lattices]%
{Lattices of Quasi-equational Theories as Congruence Lattices of 
Semilattices with Operators, Part II}

\author {Kira Adaricheva}
\address{Department of Mathematical Sciences, Yeshiva University,
New York, NY 10016, USA}
\email{adariche@yu.edu}

\author {J. B. Nation}
\address{Department of Mathematics, University of Hawaii, Honolulu, HI
96822, USA}
\email{jb@math.hawaii.edu}

\thanks{The authors were supported in part by a grant from the U.S.~
Civilian Research \& Development Foundation.  The first author was also
supported in part by INTAS Grant N03-51-4110.}

\begin{abstract}
Part I proved that for every quasivariety $\msc K$ of structures 
(which may have both operations and relations)
there is a semilattice $\S$ with operators such that the lattice 
of quasi-equational theories of $\msc K$ (the dual of the lattice of
sub-quasivarieties of $\msc K$) is isomorphic to $\op{Con}(\S,+,0,\msc F)$.
It is known that if $\S$ is a join semilattice with $0$ (and no operators), 
then there is a quasivariety $\msc Q$ such that the lattice of theories 
of $\msc Q$ is isomorphic to $\op{Con}(\S,+,0)$.  We prove that
if $\S$ is a semilattice having both $0$ and $1$ with a group $\msc G$
of operators acting on $\S$, 
and each operator in $\msc G$ fixes both $0$ and $1$,
then there is a
quasivariety $\msc W$ such that the lattice of theories of $\msc W$
is isomorphic to $\op{Con}(\S,+,0,\msc G)$.
\end{abstract}

\maketitle

\section{Introduction}

In Part I, we proved that for every quasivariety $\msc K$ of 
structures there is a semilattice $\S$ with operators such that the lattice 
of quasi-equational theories containing the theory of $\msc K$ 
is isomorphic to $\op{Con}(\S,+,0,\msc F)$.
(An \emph{operator} is a $+,0$-endomorphism.)
In this second part, we will be concerned with trying to represent
a congruence lattice $\op{Con}(\S,+,0,\msc F)$ as a lattice of 
quasi-equational theories.  This is not always possible, because the
theory generated by the equation $x \approx y$ has special properties,
with no analogue in the congruence lattice of an arbitrary semilattice
with operators.  For example, the congruence lattice of
$\mbf \Omega = (\omega,\join,0,p)$, where $p(0)=0$ and $p(x)=x-1$ for $x>0$,
is isomorphic to $\omega + 1$.  This is not a lattice of quasi-equational
theories, because it does not support an equa-interior operator; see 
Section~5 below or the remarks after Theorem~15 of Part I.

Nonetheless, there are the following positive results.
\begin{itemize}
\item Gorbunov and Tumanov proved that if $\S$ is a join semilattice with $0$, 
then there is a quasivariety $\msc Q$ such that the lattice of theories of 
$\msc Q$ is isomorphic to $\op{Con}(\S,+,0)$ with no operators;  see \cite{GT0}.
\item In this paper, we prove that if $\S$ is a semilattice having both 
$0$ and $1$ with a group $\msc G$ of operators acting on $\S$, 
and each operator in $\msc G$ fixes both $0$ and $1$, then there is a
quasivariety $\msc W$ such that the lattice of theories of $\msc W$
is isomorphic to $\op{Con}(\S,+,0,\msc G)$.
In fact, the construction works for a slightly more general class of 
operators than groups, but still a rather special type of monoid.
\item  In a third part of this study, the second author shows that 
the congruence lattice $\op{Con}(\S,+,0,\msc F)$ of a semilattice with
operators can be represented as a lattice of implicational theories in
a language that may not contain a primitive equality relation~\cite{JBN08}.
\item  A fourth part of the study looks at the structure of lattices of 
atomic theories in languages without equality~\cite{HKNT}.
\end{itemize}

Let us review, from Gorbunov~\cite{VG2} or Part~I, how congruences work
on structures with both operations and relations.
A \emph{congruence} on a structure $\mbf A = \la A, \msc F^{\mbf A},
\msc R^{\mbf A} \ra$ is a pair $\theta = \la \theta_0, \theta_1 \ra$ where
\begin{itemize}
\item $\theta_0$ is an equivalence relation on $A$ that is compatible with
the operations of $\msc F^{\mbf A}$, and
\item $\theta_1 = \bigcup_{R \in \msc R} \theta_1^R$ where, for each
relation symbol $R \in \msc R$, 
\begin{enumerate}
\item[(a)] $R^{\mbf A} \subseteq \theta_1^R \subseteq  A^{\rho(R)}$, i.e., 
the original relations of $\mbf A$ are contained in those of $\theta_1$, and 
\item[(b)] if $\mbf a \in \theta_1^R$ and $\mbf b \in A^{\rho(R)}$ and
$\mbf a \,\theta_0\, \mbf b$ componentwise, then $\mbf b \in \theta_1^R$.
\end{enumerate}
\end{itemize}

For an atomic formula $\alpha$ on a structure $\A$ and a congruence $\theta$, 
let us write $\alpha \in \theta$ to mean either 
(1) $\alpha$ is $s \approx t$ where $(s,t) \in A^2$ and $(s,t) \in \theta_0$,
or (2) $\alpha$ is $R(\mbf s)$ where $\mbf s \in A^k$ and 
$\mbf s \in \theta_1^R$.  For a quasivariety $\msc K$ of structures,
let $\op{con}_{\msc K}(\alpha)$ denote the smallest $\msc K$-congruence
on $\A$ containing $\alpha$.

Let $\S$ be the semilattice of compact $\msc K$-congruences of 
a structure $\A$.
This is a join semilattice with zero.
Every endomorphism $\vare$ of $\A$ induces an endomorphism 
$\widehat{\vare}$ of $(\S,\join,0)$, as follows. 
The endomorphism acts componentwise on $\S^k$:  
$\vare(s_1,\dots,s_k) = (\vare s_1,\dots,\vare s_k)$.
Thus we can write $\vare \alpha$ to mean either $\vare s \approx \vare t$
or $R(\vare \mbf s)$, as appropriate.
Every compact $\msc K$-congruence $\varphi$ on $\A$ can be expressed as a 
finite join $\varphi = \Join_k \op{con}_{\msc K}(\alpha_k)$.
Define $\widehat{\vare}(\varphi) = \Join_k \op{con}_{\msc K}(\vare \alpha_k)$.
It is shown in Section~3 of
Part~I that $\widehat{\vare}$ is well-defined, and preserves
joins and $0$.  Thus $\widehat{\vare}$ is an \emph{operator} on $\S$,
i.e., a $(\join,0)$-endomorphism.
In our representation of lattices of quasi-equational theories
as congruence lattices of semilattices with operators, in Part~I,
the semilattices $\S$ are the compact $\msc K$-congruences of a 
$\msc K$-free structure $\F$, and the operators are those induced
on $\msc S$ by the endomorphisms of $\F$.

\section{Representations of $\op{Con}(\S,+,0)$}

This section describes ways to represent the congruence lattice of a
semilattice, say $\L = \op{Con}(\S,+,0)$, as the lattice $\op{QTh}(\msc B)$ 
of quasi-equational theories containing a theory $\msc B$.  
Note that the congruence lattice of a semilattice is coatomistic.  
Later on, we will modify the representations to fit $\op{Con}(\S,+,0, \msc F)$ 
where $\msc F$ is a sufficiently nice set of operators.

It will be convenient to
use a closely related type of relation, rather than congruences.
For an algebra $\A$ with a join semilattice reduct, let $\op{Eon}\,\A$ be 
the lattice of all reflexive, transitive, compatible relations $R$ such that
\begin{enumerate}
\item $R \subseteq \leq$, i.e., $x \,R\, y$ implies $x \leq y$, and
\item if $x \leq y \leq z$ and $x \,R\, z$, then $x \,R\, y$.
\end{enumerate}

\begin{lm}
If\/ $\A = \la A, +, 0, \msc F \ra$ is a semilattice with operators,
then\/ $\op{Con}\,\A \cong \op{Eon}\,\A$.
\end{lm}

The isomorphism is given by the map $\theta \mapsto \theta \,\cap \leq$;
see the proof of Lemma~7 in Part~I.

The general setup is described as follows.  
We are given a semilattice $\S = \la S,+,0 \ra$, and we want to construct
a quasi-equational theory $\msc B$ such that 
$\op{QTh}(\msc B) \cong \op{Eon}\,\S$. 
Some simplifications are possible in the finite case, 
and the second representation requires that the semilattice have a 
greatest element, which will be denoted by $1$.

1.  Label $S = \{ 0,a,b,c,\dots \}$.  

2.  Construct $\op{Eon}\,\S$.
For $a<b$ in $\S$, let $\la a,b \ra$ denote the principal eon-relation
generated by $(a,b)$.  Compact eon-relations are joins of finitely many of 
these.  Coatoms of $\op{Eon}\,\S$ correspond to congruences with two
blocks, an ideal and its complement.

Eon relations of the form $\Join_{b \in I} \la 0,b \ra$ for an ideal $I$
of $\S$ will be termed \emph{equational}.  The equational relations include
the least and greatest eon-relations, $\Delta$ and $\nabla$ respectively,
and are closed under joins.

An example of $\op{Eon}\,\S$ is given in Figure~\ref{figx3}, where for 
space purposes $\la a,b \ra$ is abbreviated as $ab$.  The solid points 
indicate the equational eon-relations.

3.  The plan is now to assign predicate symbols $A$, $B$, $C, \dots$ to the
elements of $\S$, and to define a quasi-equational theory $\msc B$
in this language that will represent $\L$ as the lattice of theories
containing $\msc B$.

For each finite join $c = \Join a_j$ in $\S$,
let $C$ be identified with the conjunction $\& A_j$, so that
$C(x) \ifff \& A_j(x)$ will be part of the theory of $\msc B$.
(When $\S$ is finite, we only need predicate symbols for the join 
irreducible elements of $\S$.)
It is sometimes convenient to have special predicate symbols $T$ and $E$
reserved corresponding to $0$ and $1$, respectively.

\subsection*{First representation.}
The simplest representation, from Gorbunov and Tumanov~\cite{GT0}
(see Theorem~5.2.8 in Gorbunov~\cite{VG2}), 
has unary predicates $A(x)$, $B(x)$ etc.~for the elements of $\S$.
The quasi-equational theory $\msc B$ satisfies the laws
\begin{align*}
x &\approx y \\
A(x) &\!\implies\! B(x) \text{ whenever }a \geq b \\
\&_i \, A_i(x) &\!\implies\! B(x) \text{ whenever }\Join_i a_i \geq b .
\end{align*}
The isomorphism from $\op{Eon}\,\S$ to $\op{QTh}(\msc B)$ has $\la 0,b \ra$ 
mapping to the theory given by the law $B(x)$, and $\la a,b \ra$ mapping 
to the theory given by the law $A(x) \!\implies\! B(x)$.
More generally, $\la \Join a_i, \Join b_j \ra$ corresponds to the conjunction
over the indices $j$ of laws $\&_i A_i(x) \!\implies\! B_j(x)$, and the join 
of a set of principal eon-relations corresponds to their conjunction.

\begin{figure}[htbp]
\begin{center}
\includegraphics[height=2.8in,width=6.0in]{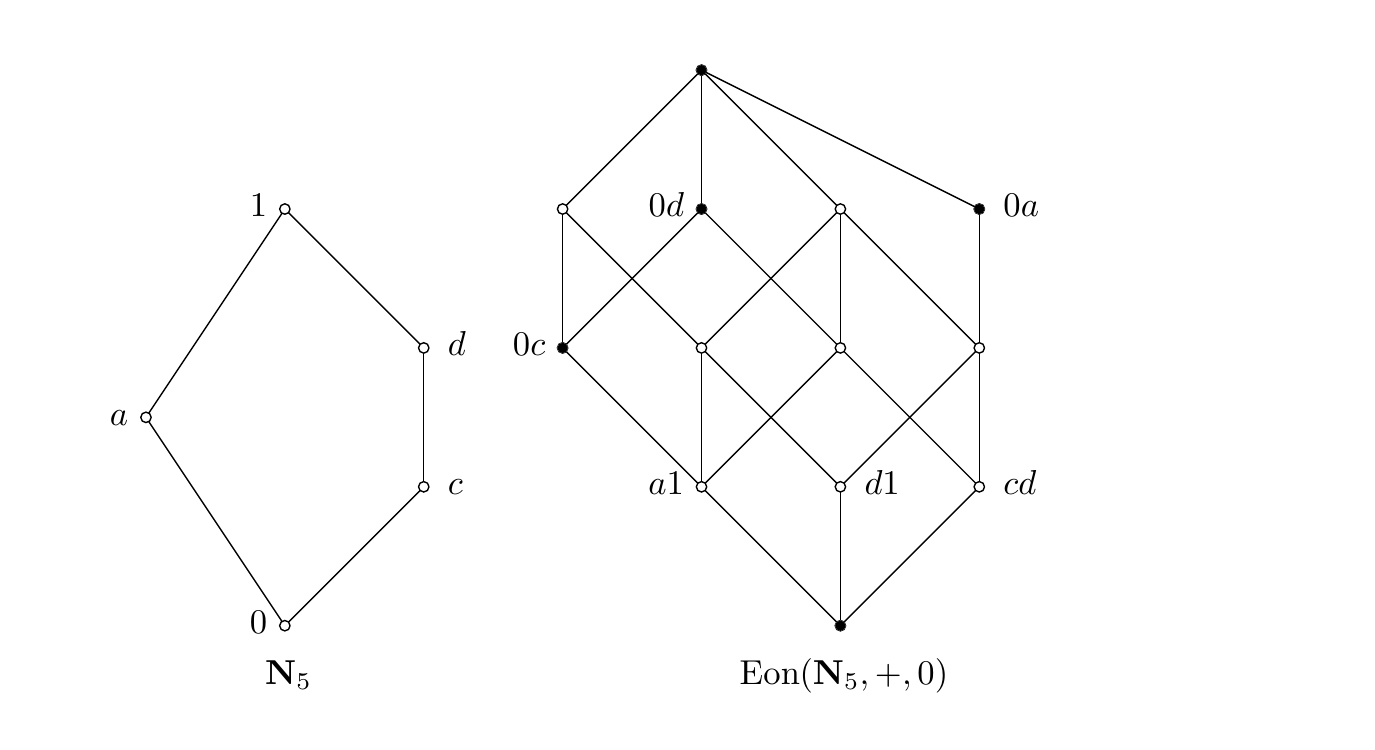}
\caption{Example of $\op{Eon}(\S,+,0)$}{\label{figx3}}
\end{center}
\end{figure}

\subsection*{Second representation.}
If we assume that $\S$ has a largest element $1$, then there is
another kind of representation of $\op{Con}(\S,+,0)$.
This incorporates some ideas from a different representation due to
Gorbunov; see Theorem 5.2.10 of Gorbunov \cite{VG2}, though it goes back 
to \cite{VG0}.

We again use unary predicates $A(x)$ corresponding to elements of $\S$
(or just the join irreducibles in the finite case),
and a predicate $E(x)$ corresponding to $1$.  Let $e$ be a constant
symbol.  The representation is given by the quasi-equational theory $\msc B$ 
with the laws
\begin{align*}
A(e) &\text{ holds for all }A \\
E(x) &\!\implies\! x \approx e \\
A(x) &\!\implies\! B(x) \text{ whenever }a \geq b \\
\&_i \, A_i(x) &\!\implies\! B(x) \text{ whenever }\Join_i a_i \geq b .
\end{align*}
Using the first two laws, we obtain the following.

\begin{lm}
Every quasi-identity of\/ $\msc B$ is equivalent to laws with at most one
variable.
\end{lm}

\begin{proof}
The atomic formulae in the language of $\msc B$ have the forms
\[ A(x) \qquad  A(e) \qquad x \approx e \qquad x \approx y \ . \]
A quasi-identity is $\&_i \lambda_i \imp \rho$ with each $\lambda_i$
and $\rho$ an atomic formula.  We may assume that no $\lambda_i$ 
is of the form $x \approx y$ or $x \approx e$.

If $\rho$ involves only a single variable $x$,
or $x$ and $e$, then we can replace all the remaining variables
by $e$, and then omit those terms to obtain a law in one variable
that is equivalent to $\lambda \imp \rho$ modulo the laws of $\msc B$.
For example, $(\&_i\, A_i(x)) \ \&\ (\&_j\, B_j(y)) \imp C(x)$
is equivalent to $\&_i\, A_i(x) \imp C(x)$ because $B_j(e)$ holds 
in $\msc B$ for all $j$.

If $\rho$ is $x \approx y$, we can first replace every variable 
except $x$ by $e$, and then remove all $\lambda_i$ of the form $B(e)$,
to obtain a law $\lambda^* \imp x \approx e$.
Then replace every variable except $y$ by $e$, and remove all $\lambda_j$ 
of the form $A(e)$, to obtain a law $\lambda^{**} \imp y \approx e$.
The two laws together are equivalent to the original, because
$\lambda^* \ \&\  \lambda^{**} \imp x \approx e \approx y$.

For example, the law
$(\&_i\, A_i(x)) \ \&\ (\&_j\, B_j(y)) \ \&\ (\&_k C_k(z)) \imp x \approx y$
is equivalent to the two laws
$\&_i\, A_i(x) \imp x \approx e$ and
$\&_j\, B_j(x) \imp y \approx e$.
\end{proof}

In $\msc B$, all the following hold.
\begin{enumerate}
\item[(i)]  Every quasi-identity is equivalent to one-variable laws.
\item[(ii)] Every predicate except $\approx$ is unary.
\item[(iii)] $x \approx e \ifff E(x)$.
\item[(iv)]  $A(e)$ holds for all predicate symbols $A$.
\end{enumerate}
So, in considering extensions of the quasi-equational theory
of $\msc B$, we may restrict our attention to implications of
the form $P \imp Q$ where $P$ and $Q$ are conjunctions of atomic 
predicates $A(x)$, or empty.  

Again, the isomorphism of $\op{Eon}\,\S$ with the lattice of quasi-equational
theories containing $\msc B$ has $\la 0,b \ra$ mapping
to the law $B(x)$, and $\la a,b \ra$ mapping to $A(x) \!\implies\! B(x)$.
More generally, $\la \Join a_i, \Join b_j \ra$ corresponds to the conjunction
over the indices $j$ of laws $\&_i A_i(x) \!\implies\! B_j(x)$.

Now one can see that the order and join dependency relation on
the semilattice of compact elements of $\op{Eon}\,\S$ are determined
as follows.
\begin{enumerate}
\item $\la a,b \ra \leq \la c,d \ra$ iff $a \geq c$ and $a \join d \geq b$,
\item $\la a,b \ra \leq \Join_j \la c_j,d_j \ra$ iff there exists a sequence
$e_1 < f_1=e_2 < f_2=e_3 < \dots <f_k$ such that
\begin{align*}
a &\geq e_1 \\
\forall i \exists j \ \la e_i,f_i \ra &\leq \la c_j,d_j \ra \\
a \join f_k &\geq b .
\end{align*}
\end{enumerate}
That is, for $a \leq b$ in $\S$, we have
$(a,b) \in \la c,d \ra$ iff $a=b$ or $a < b$, $a \geq c$, $a \join d \geq b$.
The proof is routine checking.  
Similarly, check that the description (2) of the join is correct.

Likewise, in the theory of $\msc B$, the implication $A(x) \imp B(x)$ 
is a consequence of $C(x) \imp D(x)$
if and only if $A(x) \imp C(x)$ and $A(x) \ \&\  D(x) \imp B(x)$.
A condition analogous to (2) describes when $A(x) \imp B(x)$ 
is a consequence of the conjuction of implications $C_j(x) \imp D_j(x)$.
Thus the rules for deduction in $\op{QTh}(\msc B)$ mimic
the rules for eon-relation generation in $\op{Eon}\,\S$.
Consequently, $\op{QTh}(\msc B) \cong \op{Eon}\,\S$.

\section{The dual leaf as the congruence lattice of a semilattice
with operators} \label{leaf}

The \emph{dual leaf} is the lattice in Figure~\ref{figx1}; it is the dual of
$\mbf 1 \dot+ \op{Co}(\mbf 4)$, where $\op{Co}(\mbf 4)$ is the lattice
of convex subsets of a 4-element chain.  The dual leaf is meet semidistributive
but not upper bounded, and supports an equa-interior operator, \emph{viz.},
$\eta(1)=1$ and $\eta(x)=0$ for all $x<1$.
It is an open question whether the dual leaf is a lattice of quasi-equational
theories $\op{QTh}(\msc K)$ for some $\msc K$.  However, the dual leaf
has a natural representation as $\op{Con}(\S,\join,0,f,g)$ where $\S$
is given in Figure~\ref{figx2}, $f(x_k)=x_{k+1}$ and $g(x_k)=x_{k-1}$ for
$x \in \{ a,b,c,d \}$ and $k \in \mathbb Z$.  (Compare the representation
of $\op{Co}(\mbf 4)$ as a lattice of $\vare$-closed algebraic subsets in
Example 5.5.10 of Gorbunov \cite{VG2}.)

In Section~\ref{slgo}, we will modify this example to represent the
dual near-leaf of Figure~\ref{figx7} as a lattice of quasi-equational
theories.

\begin{figure}[htbp]
\begin{center}
\includegraphics[height=2.5in,width=6.0in]{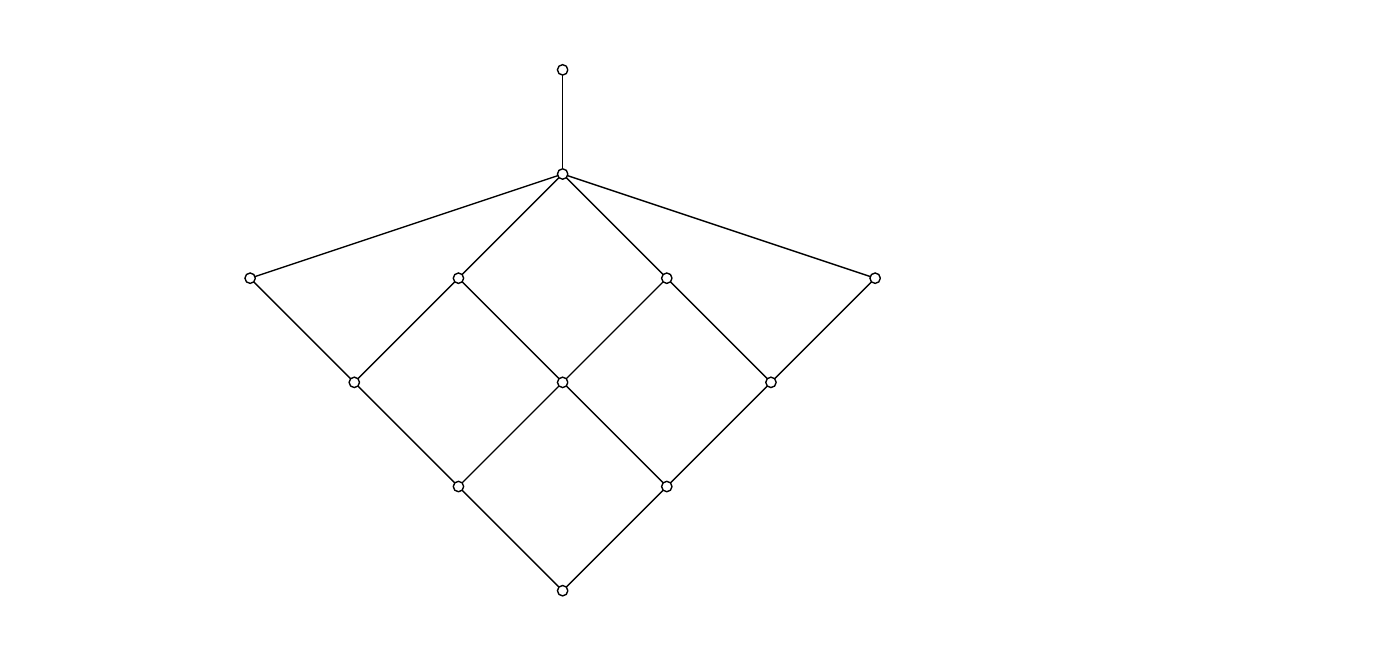}
\caption{Dual leaf}{\label{figx1}}
\end{center}
\end{figure}

\begin{figure}[htbp]
\begin{center}
\includegraphics[height=3.5in,width=6.0in]{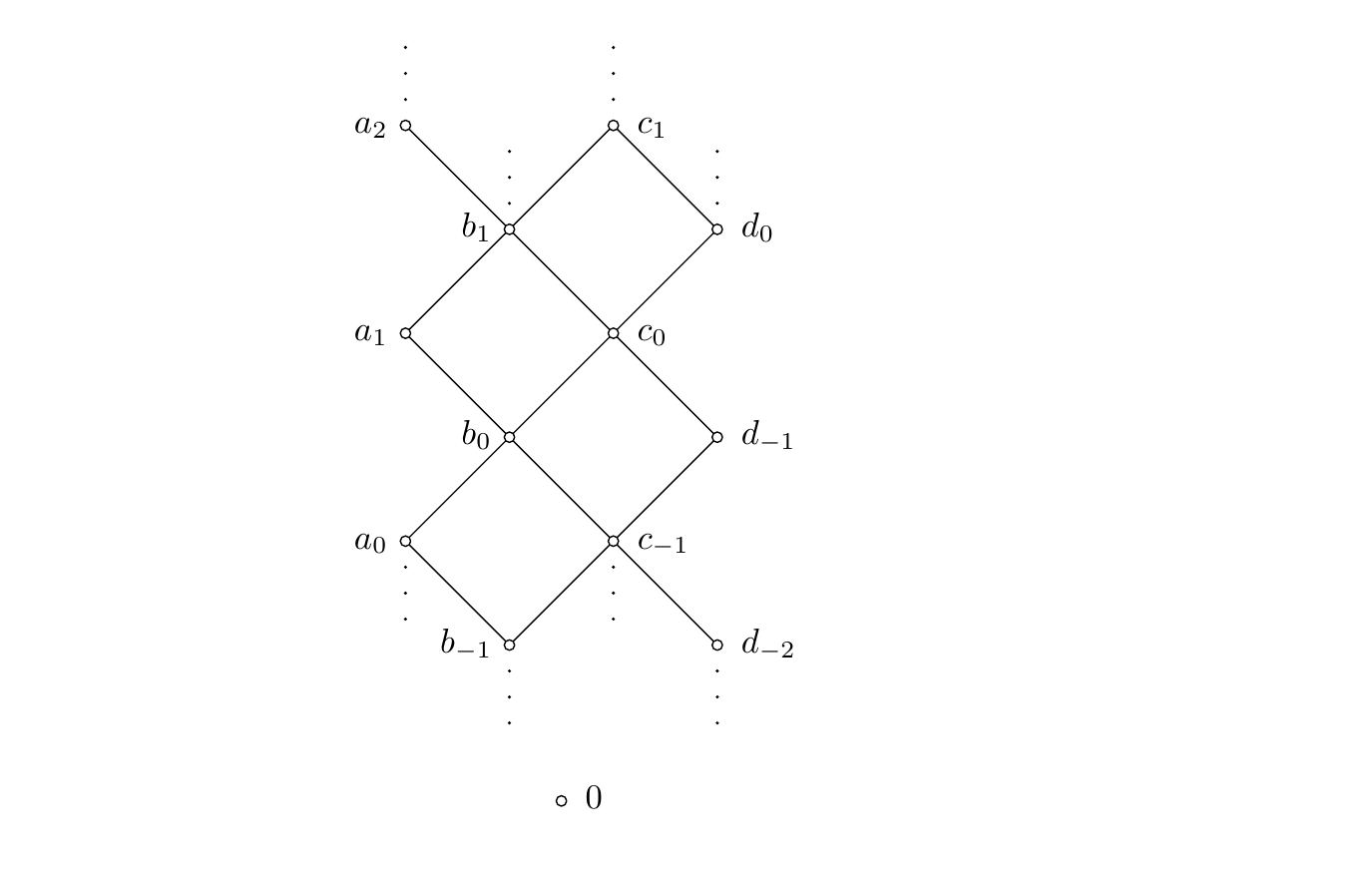}
\caption{$\S$ to represent the dual leaf}{\label{figx2}}
\end{center}
\end{figure}

\section{Sufficient conditions for reduction to one variable}

An important objective is to find conditions that allow us to represent
certain congruence lattices of semilattices with operators as lattices
of quasi-equational theories.  
The motivating case is when the semilattice has a largest element
and the operators form a group, but the results are slightly more 
general than that.  
We begin with properties that permit us to consider only 
quasi-identities in one variable.

Let us say that a monoid $\M$ is \emph{reductive} if 
\begin{enumerate}
\item[(R)] for every pair $f$, $g \in \M$ there is an element $h \in \M$
such that either $f=hg$ or $g=hf$.
\end{enumerate}
This is a rather strong property, but in particular, every group is reductive.

Let $\star$ denote the operation in $\M^{\op{opp}}$, so that $f \star g = gf$.

\begin{thm} \label{unary-rev}
Let $\msc B$ be a quasivariety in a language $\msc L$ with the following
restrictions and laws.
\begin{enumerate}
\item $\msc L$ has only unary predicate symbols (except for $\approx$).
\item $\msc L$ has only unary function symbols, corresponding to the
elements of a fixed reductive monoid $\M$.
\item $\msc L$ has one constant symbol $w$.
\item $\msc B$ satisfies the law $A(x) \imp A(w)$ for every predicate symbol 
of $\msc L$.
\item $\msc B$ satisfies the law $f(w) \approx w$ for every function symbol 
of $\msc L$.
\item $\msc B$ satisfies identities saying that the functions act as the
monoid $\M^{\op{opp}}$, that is, $i(x) \approx x$ and $f(g(x)) \approx h(x)$
whenever $h = f \star g = gf$.
\item $\msc B$ satisfies the laws
\[ f(x) \approx f(y) \!\implies\! x \approx y \]
for every function symbol of $\msc L$.
\end{enumerate}
Then every quasi-identity holding in a theory extending the theory of $\msc B$ 
is equivalent (modulo the laws of $\msc B$) to a set of quasi-identities in 
only one variable.  Hence the lattice of theories 
$\op{QTh}(\msc B)$ is isomorphic to
$\op{Con}\,\S$, where $\S = \la \T,\join,0,\widehat{\E} \ra$ with
$\T$ the semilattice of compact $\msc B$-congruences of
the $\msc B$-free structure\/ $\mbf F_{\msc B}(1)$ 
and\/ $\widehat{\E}$ the monoid of endomorphisms of\/ $\T$ induced by 
the endomorphism monoid $\op{End}(\mbf F)$.
\end{thm}

\begin{proof}
First note that, by properties (5) and (7), for each function symbol $f$, 
the quasivariety $\msc B$ satisfies the law $f(x) \approx w \ifff x \approx w$.
Secondly, each atomic formula of the form $f(x) \approx g(y)$ is equivalent 
modulo the laws of $\msc B$ 
to one of the forms $x \approx h(y)$ or $h(x) \approx y$.  
For by property (R), there is an element $h \in \M$ such that either 
$f=hg$ or $hf=g$.  If $f=hg$, then using laws (6) and (7),
\[  f(x) \approx g(y) \quad\text{iff}\quad
    (hg)(x) \approx g(y) \quad\text{iff}\quad
   g(h(x)) \approx g(y) \quad\text{iff}\quad h(x) \approx y  \]
and a similar calculation applies if $hf=g$.

Thus every atomic formula is equivalent in the theory of $\msc B$ to a 
formula of one of the following forms:
\[  A(x) \qquad A(f(x)) \qquad A(w) \qquad x \approx y \qquad x \approx f(x)
\qquad x \approx f(y) \qquad x \approx w. \]

We want to show that any law $\lambda \!\implies\! \rho$ involving more
than one variable is equivalent to a set of laws involving fewer variables.
We may assume that $\rho$ is an atomic formula of one of the above forms,
involving either a variable $x$, or variables $x$ and $y$, or the constant
$w$, or $x$ and $w$.

The premise $\lambda$ is a conjunction of atomic formulae.  
By property (5), we can replace any appearance of $f(w)$ in $\lambda$ by $w$.
By the remarks in the first paragraph of the proof, we can replace any
appearance of a formula of the form $f(x) \approx w$ or $f(x)=g(y)$ by
an equivalent formula.  Thus,
using $z$ to denote an arbitrary variable not appearing in $\rho$ 
(there may be more than one of these), then $\lambda$ is a conjunction 
involving some (including none or all) of the following forms:
\begin{align*}
&A(x) & &A(y) & &A(z) & &A(w)\\
&A(f(x)) & &A(f(y)) & &A(f(z)) \\
&x \approx y & &y \approx z & &z \approx z' \\
&x \approx z & &y \approx f(x) & &z \approx f(x) \\
&x \approx f(x) & &y \approx f(y) & &z \approx f(y) \\
&x \approx f(y) & &y \approx f(z) & &z \approx f(z) \\
&x \approx f(z) & &y \approx w & &z \approx f(z')  \\
&x \approx w & &  & &z \approx w .
\end{align*}
If perchance the hypothesis $\lambda$ includes any one of the following forms:
\begin{align*}
&x \approx y & &y \approx z & &z \approx z' \\
&x \approx z & &y \approx f(x) & &z \approx f(x) \\
&x \approx f(y) & &y \approx f(z) & &z \approx f(y) \\
&x \approx f(z) & &y \approx w & &z \approx f(z')  \\
&x \approx w & & & &z \approx w .
\end{align*}
then we can replace all occurrences of the variable on the left by the 
corresponding expression on the right,
and obtain an equivalent law with fewer variables.  So we may assume that
none of these forms appears in $\lambda$, and thus $\lambda$ involves only
these forms:
\begin{align*}
&A(x) & &A(y) & &A(z) & & A(w)\\
&A(f(x)) & &A(f(y)) & &A(f(z)) \\
&x \approx f(x) & &y \approx f(y) & &z \approx f(z) 
\end{align*}
If, after the previous substitutions, the conclusion $\rho$ involves only $x$,
or $w$, or $x$ and $w$, i.e., it has one of the forms:
\[  A(x) \qquad A(f(x)) \qquad A(w) \qquad x \approx f(x) \qquad x \approx w \]
then we may replace all the variables except $x$ by $w$. 
To see that this is equivalent,
note that $\lambda$ involves only relational forms or $t \approx f(t)$ with
$t$ a variable, while $\msc B$ satisfies the laws 
$A(t) \imp A(w)$ for every predicate symbol and
$f(w) \approx w$ for every function symbol, by properties (4) and (5). 
Thus the implication obtained by replacing the other variables by $w$ is an 
(at most) one-variable law that is equivalent to the original.  
It remains to consider the case that $\rho$ is either $x \approx y$ or 
$x \approx f(y)$, which we now assume.

Now replace all the variables except $x$ and $y$ by $w$. 
As before, the law obtained with this substitution is equivalent to 
the original.  Hence we may assume that $\lambda$ is a
conjunction of these forms:
\begin{align*}
A(x) \qquad &A(y) \qquad A(w) \\
A(f(x)) \qquad &A(f(y)) \\
x \approx f(x) \qquad &y \approx f(y) .
\end{align*}
In particular, each of these expressions involves only one of $x$, $y$ or $w$.
We can write $\lambda$ as $\lambda_x \,\&\, \lambda_y \,\&\, \lambda_w$,
where each of these is a (possibly empty) conjunction of formulae involving
only that variable.

Now replace $y$ by $w$, and then replace any occurrence of a term of the form 
$f(w)$ by $w$ in view of property (5).
This yields a law of the form $H(x,w) \!\implies\!  x \approx w$, 
which is a consequence of $\lambda \imp \rho$ and the laws of $\msc B$.
The hypothesis $H(x,w)$ is the conjunction of the terms on the LHS involving
$x$ and $w$, and terms of the form $A(w)$ replacing any previous occurrence 
of $A(y)$ or $A(f(y))$; any terms $w \approx w$ obtained by the substitution
may be omitted.  
Note that $A(y) \imp A(w)$ and $A(f(y)) \imp A(w)$ by property (4).
Thus we can write $H(x,w)$ as 
$\lambda_x \,\&\, \lambda_y^* \,\&\, \lambda_w$, where $\lambda_y^*$ involves
only $w$ and $\lambda_y \imp \lambda_y^*$.

Similarly, replacing $x$ by $w$ yields a law
$G(y,w) \!\implies\! y \approx w$ where $G(y,w)$ can be written as
$\lambda_x^* \,\&\, \lambda_y \,\&\, \lambda_w$, with $\lambda_x^*$ involving
only $w$ and $\lambda_x \imp \lambda_x^*$.
Thus we have derived two laws, in one variable each.  

Conversely, these two laws together imply the original, because
\begin{align*}
\lambda_x \,\&\, \lambda_y \,\&\, \lambda_w &\imp
\lambda_x \,\&\, \lambda_y \,\&\, \lambda_w \,\&\, 
\lambda_x^* \,\&\, \lambda_y^*\\ 
&\imp x \approx w \ \&\ y \approx w \\
&\imp \rho
\end{align*}
since $\rho$ is either $x \approx y$ or $x \approx f(y)$, and $f(w) \approx w$
in $\msc B$.

Note that any theory extending the theory of $\msc B$ includes (1)--(7),
and thus is determined by its laws in one variable.
The last statement of the theorem is then a consequence of Theorem~6 
of Part~I,  
which states that the lattice of quasi-equational theories that (1) contain
$\msc B$, and (2) are determined relative to $\msc B$ by quasi-identities in
at most $n$ variables,
is isomorphic to $\op{Con}\,\S_n$, 
where $\S_n = \la \T_n,\join,0,\widehat{\E} \ra$ with 
$\T_n$ the semilattice of compact congruences of 
$\op{Con_{\msc B}}\,\F$, 
$\E = \op{End}\,\F$, and\/ $\mbf F = \mbf F_{\msc B}(n)$.
\end{proof}

The elements of $\mbf F_{\msc B}(1)$ are $w$ and $f(x)$ for $f \in \M$.
Endomorphisms of $\mbf F_{\msc B}(1)$ are determined by the image of $x$.
The endomorphisms are the constant map $\vare_w : t \mapsto w$ 
for all $t$, and the maps $\vare_f$ for $f \in \M$ with $\vare_f(w)=w$ and 
$\vare_f(g(x)) = g(f(x)) = (fg)x$.  Note that the mapping
$f \mapsto \vare_f$ embeds $\M$ into $\op{End}\,\mbf F_{\msc B}(1)$, as
$\vare_f \vare_g = \vare_{fg}$.
This follows from the calculation 
$\vare_f \vare_g (x) = \vare_f(g(x)) =  g(f(x)) = (fg)(x) = \vare_{fg}(x)$.

\section{The pseudo-one} \label{pseudo}

We turn to the problem of trying to represent the congruence lattice of a
semilattice with operators as a lattice of quasi-equational theories.

In a lattice $\L$ such that $\L \cong \op{QTh}(\msc K)$ for some 
quasi-equational theory $\msc K$, there is an element $u$ corresponding
to the theory generated by $\msc K \cup \{ x \approx y \}$.  This element
is referred to as the \emph{pseudo-one} of $\L$.  The element $u$
is compact, satisfies $\eta(u)=u$ for the natural equa-interior operator
on $\L$, and the interval $[u,1]$ is isomorphic to the congruence lattice 
of a semilattice.  Thus the interval $[u,1]$ is coatomistic.
The existence of the pseudo-one is listed as property (I8) of an 
equa-interior operator in Part~I.

A lattice $\K$ such that $\K \cong \op{Con}(\S,+,0,\msc F)$ for some 
semilattice with operators may have no such element.
The congruence lattice of $\mbf \Omega = (\omega,\join,0,p)$, 
where $p(0)=0$ and $p(x)=x-1$ for $x>0$, is isomorphic to $\omega + 1$.  
Since $\omega + 1$ contains no element satisfying the properties of the 
pseudo-one, it is not isomorphic to $\op{QTh}(\msc K)$ for any $\msc K$. 
In trying to represent the congruence lattice of a semilattice with operators
as a lattice of quasi-equational theories, when possible, 
one must find an element $k \in \S$ such that $\op{con}(0,k)$
is the pseudo-one of $\op{Con}(\S,+,0,\msc F)$.
Note that in the case when $\S$ has a top element $1$, 
the largest congruence $\op{con}(0,1)$ is compact, and thus is a candidate
for the pseudo-one in a potential representation of $\op{Con}\,\S$ as a 
lattice of quasi-equational theories.

Let $\msc K$ be a quasivariety with $\F = \F_{\msc K}(X)$.
Consider the least $\msc K$-congruence $\Upsilon$ on $\F$ with 
$|\F/\Upsilon|=1$, i.e., $\Upsilon = \la \nabla,\Upsilon_1 \ra$ 
where $\nabla = F \times F$ is the universal relation.  
(Congruences on structures are reviewed in the Introduction.)
Now $\Upsilon$ may or may not be compact in $\op{Con}_{\msc K}(\F)$. 
Nonetheless, it has a very nice property. 

\begin{lm}
If\/ $\theta$ is a compact $\msc K$-congruence on $\msc F$ 
and $\vare$ an endomorphism, then 
$\widehat{\vare}(\theta) \join \Upsilon = \theta \join \Upsilon$ 
in $\op{Con}_{\msc K}(\F)$.
\end{lm}

\begin{proof}
Let $\theta = \la \theta_0,\theta_1 \ra$ be a compact 
$\msc K$-congruence on $\F$.  If $(s,t) \in \theta_0$, then 
$(\vare s,\vare t) \in \nabla = \Upsilon_0$,
so $\op{con}_{\msc K}(\vare s,\vare t) \leq \theta \join \Upsilon$.
Likewise, if $\mbf s \in \theta_1^R$ for a relation $R$, then
$\vare\mbf s \,\Upsilon\, \mbf s$ componentwise, and hence
$\vare \mbf s \in (\theta \join \Upsilon)_1^R$.
Thus $\op{con}_{\msc K}(R(\vare \mbf s)) \leq \theta \join \Upsilon$.
Express $\theta$ as a join of principal $\msc K$-congruences, 
say $\theta = \Join \op{con}(\alpha_i)$. 
From the preceding arguments, we conclude that
$\widehat{\vare}(\theta) = \Join \op{con}(\vare\alpha_i) 
\leq \theta \join \Upsilon$.  The reverse inclusion, that 
$\theta \leq \widehat{\vare}(\theta) \join \Upsilon$, is similar.
\end{proof}

Now, as usual, let $\S = \la \T,\join,0,\widehat{\E} \ra$ with
$\T$ the semilattice of compact $\msc K$-congruences of 
$\F = \F_{\msc K}(X)$ 
and $\widehat{\E}$ the monoid of endomorphisms of $\T$ induced by 
the endomorphism monoid $\op{End}(\mbf F)$.
Assuming that $|X| \geq 2$, Let $\kappa$ be the compact 
$\msc K$-congruence $\op{con}_{\msc K}(x,y)$ on $\F$.  Then
\[ \Upsilon = \Join_{\vare \in \E} \widehat{\vare}(\kappa) \]
in $\op{Con}_{\msc K}(\F)$.
When $X = \{ x \}$ and the language of $\msc K$ has a constant 
symbol $w$, the same equation holds with  
$\kappa = \op{con}_{\msc K}(x,w)$.  

Back in $\S$, Lemma~\ref{pseudo} translates as follows:
for every compact $\msc K$-congruence $\theta$ of $\msc F$, 
and every $\vare \in \op{End}\,\F$, there exist endomorphisms
$\gamma_1, \dots, \gamma_k$ such that 
\[
\widehat{\vare}(\theta) \join 
\widehat{\gamma}_1(\kappa) \join \dots \join \widehat{\gamma}_n(\kappa)
= \theta \join 
\widehat{\gamma}_1(\kappa) \join \dots \join \widehat{\gamma}_n(\kappa).
\]
Thus we have the following conclusion.

\begin{thm} \label{pseudoprop}
If\/ $\L \cong \op{QTh}(\msc K)$ for some quasi-equational theory 
$\msc K$, then there is a semilattice with operators
$\S = (S,+,0,\M)$ such that\/ $\L \cong \op{Con}\,\S$, with $\S$ 
satisfying this property:  there exists an element $k \in \S$ such that,
for every $s \in \S$ and $f \in \M$, there exist finitely many
$g_1, \dots, g_n \in \M$ such that
\[  f(s)+g_1(k)+ \dots +g_n(k) = s+g_1(k)+ \dots +g_n(k). \]
\end{thm}

It is not clear how to use this property to build 
representations in a general setting.
In the next section, we will take the easy way out and 
assume that $\S$ has a largest element $1$ that is fixed
by the operators of $\M$.

\section{Lattices of 1-variable quasi-equational theories} \label{repby1}

The following result represents certain congruence lattices of 
semilattices with operators as the lattice of 1-variable quasi-equational 
theories of a quasivariety, when the semilattice has a largest element $1$
that is fixed by every operator.  In this case, $1$ can play the role of 
$\kappa$.

A monoid $\M$ is said to be \emph{right cancellative} if 
$gf=hf$ implies $g=h$ for all $f$, $g$, $h \in \M$. 
Groups, of course, are right cancellative.

\begin{thm} \label{combined}
Let $\S$ be a join semilattice with $0$ and $1$, and let $\M$ be a 
right cancellative monoid of operators acting on $\S$.  
Assume that $f(1)=1$ for every $f \in \M$.
Then there is a quasivariety $\msc C$ such that $\op{Con}(\S,+,0,\M)$
is isomorphic to $\op{Con}(\T,\join,0,\widehat{\E})$, where
$\T$ is the semilattice of compact $\msc C$-congruences of\/ 
$\mbf F_{\msc C}(1)$ and\/ $\widehat{\E}$ is the monoid of endomorphisms 
of\/ $\T$ induced by $\op{End}(\mbf F)$.
\end{thm}

Moreover, the laws of the quasivariety $\msc C$ in this theorem 
will include those of the quasivariety $\msc B$ from 
Theorem~\ref{unary-rev}.  

\begin{proof}
Our language will include unary predicate symbols $A$ for each nonzero 
element $a$ of $\S$, operation symbols $f$ for each $f \in \M$, and a 
constant $w$.  The predicate symbol $U$ corresponds to the element $1$.

The construction begins by assigning a set $\msc P(s)$
of atomic formulae to each element of $\S$, including $s=0$.  
For $s \in S$, let 
$\msc P(s) = \{ A(f(x)): a \ne 0 \text{ and }f(a) = s \}$.  
Also, let $\msc Q$ be the set of atomic formulae involving $w$ 
given by $\msc Q = \{ A(w): a \ne 0 \}$.

Define $\msc C$ to be the quasivariety determined by these laws.
\begin{enumerate}
\item[(5)] $f(w) \approx w$ for every $f \in \M $.
\item[(6)] $i(x) \approx x$ and $f(g(x)) \approx h(x)$ whenever
$h = f \star g = gf$.
\item[(7)] $f(x) \approx f(y) \imp x \approx y$ for every $f \in \M$.
\item[(8)] $f(x) \approx g(x) \imp x \approx w$ for each pair  
$f \ne g \in \M$.
\item[(9)] $U(x) \imp x \approx w$.
\item[(10)] $A(f(x))$ whenever $a \ne 0$ and $f(a)=0$ in $\S$.
\item[(11)] $A(w)$ for all nonzero $a \in \S$.
\item[(12)] $\beta \imp \alpha$ whenever $a \leq b$, $\alpha \in
\msc P(a)$, $\beta \in \msc P(b)$.
\item[(13)] $\& \beta_j \imp \alpha$ whenever $a \leq \sum b_j$,
$\alpha \in \msc P(a)$, $\beta_j \in \msc P(b_j)$ for each $j$.
\end{enumerate}
The laws of $\msc C$ contain the laws of the quasivariety $\msc B$ of 
Theorem~\ref{unary-rev}.
The language is specified to satisfy (1)--(3),
while Law (4) has been replaced by (11), which is stronger.
Laws (5)--(7) are included here, and the last six (8)--(13) are new. 
Note that laws (10) and (11) correspond to the predicates in $\msc P(0)$
and $\msc Q$, respectively.  
Also, law (12) is redundant as a special case of (13).
As in the proof of Theorem~\ref{unary-rev},
$\msc C$ satisfies $f(x) \approx w \ifff x \approx w$.


The universe of $\mbf F = \mbf F_{\msc C}(1)$ is 
$\{ f(x): f \in  \M \} \cup \{ w \}$.
The operations correspond to elements of $\M$, but
composing as in $\M^{\op{opp}}$ per law (6). 
There is a unary predicate $A^{\F}$ for each nonzero element $a$ of $\S$.
In the free structure only (10) and (11) apply, with the remaining laws adding
no additional relations.  Thus $A^{\F}(f(x))$ holds for all
$A(f(x)) \in \msc P(0)$, and $A^{\F}(w)$ holds for all $A(w) \in \msc Q$. 

The endomorphisms of $\F$ are again the constant map $\vare_w$ with 
$\vare_w(f(x)) = \vare_w(w)=w$ for all $f \in \M$, and 
the maps $\vare_f$ with $\vare_f(w)=w$ and 
$\vare_f(g(x))=g(f(x))=(fg)x$ for all $f$, $g \in \M$.

Recall that a congruence $\theta$ on a structure is a pair 
$\la \theta_0,\theta_1 \ra$ where $\theta_0$ is an equivalence 
relation and $\theta_1$ is a collection of predicates.
In describing congruences on $\F$, we will use $\Delta$ to denote the diagonal
(equality) equivalence relation, and $\nabla$ to denote the universal binary 
relation, $\nabla = F \times F$.

Note that modulo the laws of $\msc C$, every relational atomic formula is 
equivalent to one of the form $A(f(x))$.  For, using law (6), $A(x)$ is 
equivalent to $A(i(x))$, and $A(f(g(x)))$ is equivalent to $A(h(x))$ 
where $h=gf$.
Each atomic formula of this form is in exactly one set $\msc P(s)$,
\emph{viz.}, for $s=f(a)$, and so the sets $\msc P(s)$ form a partition 
of these representative atomic formulae.

{\bf Claim 1:}  The following are $\msc C$-congruences on $\F$:
\begin{itemize}
\item for any nonempty proper ideal $I$ of $\S$, 
$\theta^I = \la \Delta, \msc Q \cup \bigcup_{s \in I} \msc P(s) \ra$,
\item for the largest ideal $I = S$, 
$\theta^S = \la \nabla, \msc Q \cup \bigcup_{s \in S} \msc P(s) \ra$.
\end{itemize}
Note that $A(f(x)) \in (\theta^I)_1$ if and only if $f(a) \in I$.
For a principal ideal $\id s$, let us write $\theta^s$ instead of 
$\theta^{\id s}$.

Clearly equality and the universal relation respect the operations
of $\F$, join and all $f \in \M$.
Any set of relations is compatible with the diagonal $\Delta$, 
and because $(\theta^S)_1$ contains all possible predicates on $\F$, 
it is compatible with $\nabla$.

It remains to verify that the quotient structures $\F/\theta^I$
satisfy the laws of $\msc C$.  We may assume that $I<S$, since
$\theta^S$ is the largest congruence of $\F$.
Laws (5) and (6) follow from the definition
of the operations in $\F$, and laws (10) and (11) from the relations
$\msc R^{\F}$ holding in $\F$.

For law (7), suppose $f(g(x)) = f(h(x))$.  Then $(gf)x=(hf)x$,
and since $\M$ is right cancellative and $\F$ is free, $g=h$. 
Hence $g(x)=h(x)$.  

For law (8), $f(x) \ne g(x)$ in $\F$ for $f \ne g$, although
$f(w)=w=g(w)$.

For law (9), if $U(f(x))$ is in $(\theta^I)_1$, then 
$f(1) = 1$ is in $I$, and hence $(x,w) \in (\theta^I)_0 = \nabla$.

Since law (12) is a special case of law (13), we consider (13).
Again, we may assume $I<S$.   Let $a \leq \sum b_j$,
$\alpha = D(g(x)) \in \msc P(a)$, and 
$\beta_j = C_j(f_j(x)) \in \msc P(b_j)$ for each $j$.
Suppose $\beta_j \in (\theta^I)_1$ for each $j$, so that
$f_j(c_j)=b_j \in I$.   Then $\sum b_j \in I$, and hence 
$g(d) = a \in I$.  It follows that $\alpha \in (\theta^I)_1$,
as desired.

Thus each $\F/\theta^I$ is in $\msc C$.

{\bf Claim 2:}  
For ideals $I$ and $J$, 
$\bigcup_{s \in I} \msc P(s) \subseteq \bigcup_{t \in J} \msc P(t)$
if and only if $I \subseteq J$.  In particular, if $I \not\subseteq J$
with say $a \in I-J$, then $A(x) \in (\theta^I)_1 - (\theta^J)_1$.
Thus the map $I \mapsto \theta^I$ is one-to-one.

{\bf Claim 3:}  Every $\msc C$-congruence on $\F$ is $\theta_I$ for 
some ideal $I$.
Indeed, given a $\msc C$-congruence $\psi = \la \psi_0,\psi_1 \ra$, 
by laws (7) and  (8) its equivalence $\psi_0$ is either $\Delta$ or 
$\nabla$.  
Moreover, $\psi_1$ contains $\msc Q \cup \msc P(0)$ by laws (10)
and (11).  Let $I = \{ s \in \S : \msc P(s) \subseteq \psi_1 \}$.
Then $I$ is an ideal by (12) and (13).
If $1 \in I$, then $\psi_0 = \nabla$ by law (9), 
for $U(f(x)) \in \psi_1$ implies $(f(x),w) \in \psi_0$,
which in $\msc C$ is equivalent to $(x,w) \in \psi_0$.
If $1 \notin I$, then $\psi_0 = \Delta$, again by (7) and (8).
It follows that $\theta=\theta^I$.
Note that $\theta^I$ is compact exactly when $I$ is principal.

{\bf Claim 4:}  For $s \in \S$, $\widehat{\vare}_w(\theta^s) = 
\theta^0$, the least congruence on $\F$.
By definition, $\widehat{\vare}_w(\theta^s)$ is the $\msc C$-congruence
generated by all pairs $(\vare_w u,\vare_w v)$ with $(u,v) \in (\theta^s)_0$
and all relations $A(\vare_w u)$ with $A(u) \in (\theta^s)_1$.
The former are just the single pair $(w,w)$,
and the latter are all contained in $\msc Q$.
Thus $\widehat{\vare}_w(\theta^s) = \theta^0$, the least congruence
on $\F$. 

{\bf Claim 5:}  For an endomorphism $\vare_h$ with $h \in \M$, 
$\widehat{\vare}_h(\theta^s) = \theta^{h(s)}$.
Again, $\widehat{\vare}_h(\theta^s)$ is the $\msc C$-congruence
generated by all pairs $(\vare_h u,\vare_h v)$ with $(u,v) \in (\theta^s)_0$
and all relations $A(\vare_h u)$ with $A(u) \in (\theta^s)_1$.
If $s=1$, then $U(x)$ is one of the latter, whence 
$\widehat{\vare}_h(\theta^s) = \theta^{1}$, while $h(s)=h(1)=1$
by assumption on $\M$.  

Hence we may assume that $s<1$, and so $(\theta^s)_0 = \Delta$.
Of course, $\vare_h(\Delta) \subseteq \Delta$.  
Consider a relation $A(f(x)) \in (\theta^s)_1$.
Then $A(f(x)) \in \msc P(t)$ for some $t$ with $f(a) = t \leq s$.
Now $\vare_h(f(x)) = f(h(x)) = (hf)(x)$ in $\F$,
while $(hf)(a)=f(f(a)) = h(t) \leq h(s)$ in $\S$.
Thus $A(\vare_h(f(x)))$ is in $\msc P(h(s))$, which is 
contained in $(\theta^{h(s)})_1$.  
Moreover, with $f=i$ and $a=s$, these relations include $S(h(x))$, 
which is a generator for the $\msc C$-congruence $\theta^{h(s)}$.  
Thus $\widehat{\vare}_h(\theta^s) = \theta^{h(s)}$.

\medskip

We have seen that the compact $\msc C$-congruences of $\F$ are in
one-to-one correspondence with the elements of $\S$, via the map
$s \mapsto \theta^s$.  The action of the operators $\widehat{\vare_h}$, 
for $h \in \M$, on the semilattice $\T$ of compact $\msc C$-congruences 
of $\F$, mimics the action of $\M$ on $\S$, while $\widehat{\vare_w}$
is the constant zero map.  
Therefore $\op{Con}(\T,\join,0,\widehat{\msc E}) \cong \op{Con}(\S,+,0,\M)$.

This completes the proof of the theorem.
\end{proof}

For groups of operators, we can get by with only one predicate per orbit,
and we could use inverses instead of $\G^{\op{opp}}$.

\section{Semilattices with groups of operators} \label{slgo}

To combine the previous two theorems, we want to consider a quasivariety
$\msc D$ satisfying all the laws (1)--(13) with $\M$ both reductive and
right cancellative.

\begin{cor} \label{reductive}
Let $\S$ be a join semilattice with $0$ and $1$, and let $\M$
be a reductive, right cancellative monoid of operators acting on $\S$.  
Assume that $f(1)=1$ for every $f \in \M$.
Then there is a quasivariety $\msc D$ such that the lattice of 
quasi-equational theories of\/ $\msc D$ is isomorphic to $\op{Con}(\S,+,0,\M)$.
\end{cor}

This corollary applies to the following situations, so that in each case
the congruence lattice of the semilattice with operators is representable 
as a lattice of quasi-equational theories.
\begin{itemize}
\item $\S$ is a semilattice with 0 and 1 and a group of operators
fixing both $0$ and $1$.
\item In particular, $\S$ is a semilattice with 0 and 1 and $\M$ is any
subgroup of $\op{Aut}\,\S$. 
\item $\S = -\mathbb N \cup \{ -\infty,\infty \}$, where $-\mathbb N$ denotes 
the non-positive integers and $-\infty$, $\infty$ are new least and greatest
elements, respectively.  This is a join semilattice with the operators 
$p_k(x)=k+x$ for $k$, $x \in -\mathbb N$ and $p_k(\pm \infty)=\pm \infty$.
\item $\mathbb N$ can be replaced by non-positive rationals, or reals,
in the preceding example.
\item More generally, we can take $-\mbf P$ to be the negative cone of 
a totally ordered group, with new least and greatest elements $-\infty$, 
$\infty$ adjoined, and say the left translations, $p_k(x)=kx$ for 
$k$, $x \in -\mbf P$ and $p_k(\pm \infty)=\pm \infty$, as operators.
\item The operations can be restricted to a submonoid, so long as the 
reductive property is maintained.  In particular, we can restrict to
a cyclic or quasicylic monoid.
\item These representations can be combined as follows.  Let $\S$ be a
fixed semilattice with 0, 1, and a group $\G$ of operators fixing
$0$ and $1$.
Let $-\mbf P$ be a negative cone as above, and let $\K$ denote its operators.  
Now replace each point of $-\mbf P$ by $\S$, i.e., take 
$\mbf Q=-\mbf P \times \S$ with the lexicographic order, and adjoin 
new elements $\pm \infty$.  
Then $\K \times \G$ operates on $\mbf Q$ naturally:  $(p_k,f)(x,y) =
(kx,f(y))$ and $(p_k,f)(\pm \infty)=\pm \infty$.  
Moreover, it is reductive and cancellative, and so the corollary applies.  
\item If the chain $-\mbf P$ is discrete, we could identify the 0 and 1 
of consecutive semilattices in the chain in the previous construction.
\end{itemize}
There are a lot of details to be checked there, but they are routine.

For example, with the negative cone $-\mbf P$ of a totally ordered group
and its full complement of operators $p_k$, we obtain a representation
of the lattice $1 \dot+ (\msc I^*(-\mbf P \dot+ 1)) \times \mbf 2$,
where $\msc I^*$ denotes nonempty ideals, as a lattice of quasi-equational
theories.  With $-\mbf P = -\mathbb N$, we see that 
$1 \dot+ (\omega^* \times \mbf 2)$ is $\op{QTh}(\msc K)$ for some $\msc K)$.

The Universal Algebra Calculator of Ralph Freese and Emil Kiss \cite{UAC}
has proved to be useful in calculating the congruence lattice of a 
finite semilattice with a group of operators.  Even when $\S$ is fairly
small, its congruence lattice can be rather large.

Now let us represent the dual near-leaf of Figure~\ref{figx7}.
For $\S$ we take the semilattice in Figure~\ref{figx2} and add a
top element $1$.  There is a natural action of the integers $\mathbb Z$ on 
$\S$ as noted in Section~\ref{leaf}.  With the new $1$ added, we obtain the
dual near-leaf as $\op{Con}(\S,+,0,\mathbb Z)$.  Following the prescription
given above, this is the lattice of theories of the quasivariety satisfying
the laws below.  As before, we denote the action of $\mathbb Z$ by
$f(x_k)=x_{k+1}$ and $g(x_k)=x_{k-1}$ for $x \in \{ a,b,c,d \}$ and
$k \in \omega$, and $f(0)=0=g(0)$ and $f(1)=1=g(1)$.
\begin{align*}
&fg(x) \approx x \text{ and } gf(x) \approx x \\
&A(e) \qquad B(e) \qquad C(e) \qquad D(e) \\
&f(e) \approx e \qquad\quad  g(e) \approx e \\
&D(x) \!\implies\! C(x) \!\implies\! B(x) \!\implies\! A(x)\\
&C(x) \!\implies\! D(g(x)) \\
&B(x) \!\implies\! C(g(x)) \\
&A(x) \!\implies\! B(g(x)) \\
&A(x) \ \&\ C(g(x)) \!\implies\! B(x) \\
&B(x) \ \&\ D(g(x)) \!\implies\! C(x) \\
&x \approx f^k(x) \!\implies\! x \approx e \quad \text{ for all }k > 0 \\
&x \approx g^k(x) \!\implies\! x \approx e \quad \text{ for all }k > 0.
\end{align*}

\begin{figure}[htbp]
\begin{center}
\includegraphics[height=3.0in,width=6.0in]{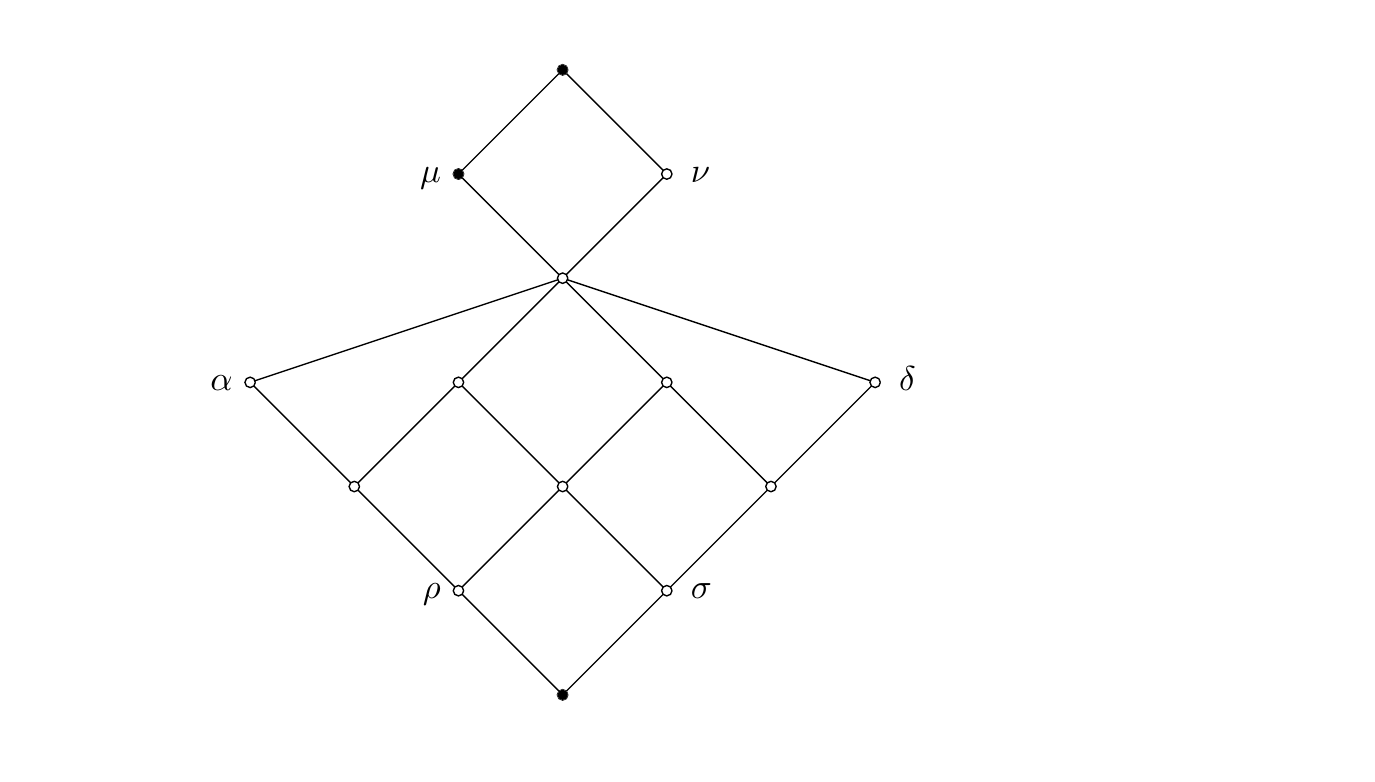}
\caption{Dual near-leaf}{\label{figx7}}
\end{center}
\end{figure}

Again, it is instructive to make a chart of the dual near-leaf with
the corresponding congruence generators and theories (where $\sim$ indicates
that laws are equivalent modulo the defining relations).
Here are some of them:
\begin{align*}
&\alpha  &&\la a_k,b_k \ra \join \la b_k,c_k \ra \join \la c_k,d_k \ra
&&A(x) \!\RA\! B(x) \!\RA\! C(x) \!\RA\! D(x) \\
&\delta  &&\la d_k,c_{k+1} \ra \join \la c_{k+1},b_{k+2} \ra \join \la
b_{k+2},a_{k+3} \ra &&D(x) \!\RA\! C(fx) \!\RA\! B(f^2x)
\!\RA\! A(f^3x) \\
&\rho  &&\la d_k,c_{k+1} \ra &&D(x) \!\implies\! C(f(x)) \\
&\sigma  &&\la a_k,b_{k-1}\ra  &&A(x) \!\implies\! B(g(x)) \\
&\mu  &&\la 0,a_k \ra = \la 0,b_k \ra = \la 0,c_k \ra =
\la 0,d_k \ra &&A(x) \sim B(x) \sim C(x) \sim D(x) \\
&\nu &&\la a_k,1 \ra = \la b_k,1 \ra = \la c_k,1 \ra = \la d_k,1 \ra
&& A(x) \!\RA\! x \approx e \sim \dots \sim D(x) \!\RA\!
x \approx e
\end{align*}
Note that only $0$, $1$ and $\mu$ are equational.

The dual near-leaf is not an upper bounded lattice, and thus answers in
the negative Question~3 from Adams, Adaricheva, Dziobiak and Kravchenko
\cite{AADK}.  Dually, not every finite $Q$-lattice is lower bounded.

\section{Discussion}

The constant $w$ played a crucial role in the reduction to one variable,
Theorem~\ref{unary-rev}.  This was borrowed from Gorbunov, see
\cite{VG2}, and some such device is necessary.  For example, the law
$A(x) \!\implies\! B(y)$ says that if there exists an $x$ with $A(x)$,
then $B(y)$ holds for all $y$.  It cannot be reduced to one variable
without introducing a constant.

On the other hand, the properties that we have
given to $w$ have the consequence of making
the unit congruence of $\mbf F_{\msc C}(1)$ compact, so that we must 
limit consideration to semilattices with both 0 and 1.  
In turn, if $\S$ has both 0 and 1, then the largest congruence
of $\op{Con}(\S,+,0,\msc F)$ is compact.

\begin{lm}
Let\/ $\S = (S,+,0,\msc F)$ be a semilattice with a monoid $\msc F$ 
of operators.  The largest congruence of\/ $\S$ is compact if and only if\/
$\msc F u = \{f(u): f \in \msc F \}$ is cofinal in $\S$ for some $u \in S$.
\end{lm}

\emph{Question:}  Under what circumstances is it true that if the largest
congruence of $\S$ is compact, then $\op{Con}\,\S \cong \op{Con}\,\T$
for some $\T$ with both 0 and 1?

Now consider a semilattice with operators $\S$ that has 0 and 1.  
Let $\theta$ be any congruence
on $\S$, and let $F$ be the order-filter $1/\theta$.  Then $F$ satisfies
\[ (*) \quad (\forall f \in \msc F)(\forall a,b \in F)(\forall s \in S)
   \  fa+s \in F \!\implies\! fb+s \in F  .\]
For any such order filter $F$, there is an interval $[\varphi(F),\psi(F)]$
in $\op{Con}\,\S$ of congruences $\theta$ such that $F=1/\theta$.
The congruences $\varphi(F)$ and $\psi(F)$ can be described thusly.
\begin{itemize}
\item $(x,y) \in \varphi(F)$ if $x=y$ or there are a sequence $x=x_0, x_1,
\dots, x_n=y$, operators $f_i \in \msc F$, elements $a_i$, $b_i \in F$
and $s_i \in S$ such that 
\begin{align*}
   x_i &= f_ia_i+s_i \\
   x_{i+1} &= f_ib_i+s_i .
\end{align*}
\item $(x,y) \in \psi(F)$ if for all $f \in \msc F$ and all $s \in S$,
\[   fx+s \in F \quad\text{iff}\quad fy+s \in F .\]
\end{itemize}
As before, these functions can be extended to $\varphi(x)$ and $\psi(x)$,
and $\varphi$ is an interior operator on $\op{Con}\,\S$.  These operations 
should be useful, but at this point it is not clear how.

\section{Summary and questions to pursue}

Our results can be reasonably summarized, and compared to previous results
on lattices of quasivarieties, by considering four classes of lattices.
\begin{itemize}
\item $\msc Q$ is all $Q$-lattices $L_q(\msc K)$ for a quasivariety $\msc K$.
\item $\msc C^d$ is all duals of $\op{Con}\,(\S,+,0,\msc F)$ with $\msc F$
a set of operators.
\item $\msc S$ is all $\op{Sp}(\A,\vare)$ with $\A$ algebraic and $\vare$
a Brouwerian, filterable, continuous quasi-order.
\item $\msc J$ is all join semidistributive, atomic, dually algebraic 
lattices supporting an equaclosure operator satisfying the duals of 
conditions (I1)--(I9) from Part~I.
\end{itemize}
Each of the latter three classes contains $\msc Q$, and thus
provides a different type of description of lattices of quasivarities.  
Part~I established the inclusion $\msc Q \subset \msc C^d$, and derived
some of its consequences.
This interpretation led to the identification of new properties of the natural
equaclosure operator on $\L_q(\msc K)$.

This paper has focused on the reverse problem of representing congruence 
lattices of semilattices with operators as lattices of quasi-equational 
theories.
\begin{itemize}
\item The representation of the congruence lattice $\op{Con}(\S,+,0)$
of a semilattice with no operators as a lattice of quasi-equational theories 
becomes quite transparent with this viewpoint.
\item Every congruence lattices $\op{Con}(\S,+,0,\msc G)$, where $\S$ has 
both $0$ and $1$ and $\msc G$ is a group of operators fixing both $0$ and $1$, 
can be represented as a lattice of quasi-equational theories.
\item There is a finite lattice of quasi-equational theories that is not
an upper bounded lattice.
\end{itemize}
A gap in our current understanding is that we have not found an effective
way to deal with condition (I8) for equa-interior operators, the existence 
of a pseudo-one, which can fail in the congruence lattice of a semilattice 
with operators.  

The older representations of $Q$-lattices as lattices of algebraic sets
had problems with the duals of conditions (I8) and (I9) from Part~I.  
In that sense, the new representation may be preferable.  
But it is not clear at all that $\msc C^d$ and $\msc S$ are comparable:
we don't know if the relation determining the congruence lattice of a 
semilattice $\S$ with operators as a complete sublattice of 
$\op{Sp}(\op{Con}\,\S)$ is continuous.  (See Appendix I of Part I.)
Since continuity is the only issue, it is true that if $\S$ is finite,
then the dual of $\op{Con}(\S,+,0,\msc F)$ is in $\msc S$.

There is no reason to think that the properties describing $\msc J$
actually characterize $Q$-lattices, but we believe that they summarize
what is known at this point.  That leaves us with some interesting 
questions.

\begin{enumerate}
\item Given a lattice $\L = \op{Con}(\S,+,0,\msc F)$, when can we represent
$\L$ as the lattice of theories of a quasivariety?  
\item In particular, can we represent the dual leaf (Figure~\ref{figx1}) as 
the lattice of theories of a quasivariety?
\item Given a finite meet semidistributive lattice $\L$ with an equa-interior
operator, when can we represent $\L$ as either the congruence lattice
of a semilattice with operators or the lattice of theories of a
quasivariety?   
\item Find an algorithm to determine whether a finite meet semidistributive
lattice supports an equa-interior operator satisfying (I1)--(I9).  
(We have done this for the original conditions; see~\cite{AN}.)
\item We know that the variety given by the law $x \approx y$ plays a special
role in the lattice of quasivarieties.  Find a good description of this 
behavior in the context of semilattices with operators.
\end{enumerate}

The authors would again like to thank the referee for many 
helpful comments.

\end{document}